\documentclass[reqno]{amsart}
\pdfoutput = 1

\usepackage[utf8x]{inputenc}
\usepackage[T1]{fontenc}
\usepackage{lmodern}
\usepackage[english]{babel}
\usepackage[draft=false,kerning=true]{microtype}

\usepackage{amsmath, amsfonts, amssymb, amsthm, amstext, bbm}
\usepackage{euscript}

\usepackage{tikz}
\usetikzlibrary{calc,shapes}
\usepackage{pgfplots}

\usepackage{ifthen,calc}
\usepackage{suffix}
\usepackage{enumerate}
\usepackage{nameref}
\usepackage{hyperref}

\newcommand{\parentheses}[4][]%
{\mathopen{}\ifthenelse{\equal{#1}{}}{\left#2}{\csname#1\endcsname#2}%
    {#4}\mathclose{}\ifthenelse{\equal{#1}{}}{\right#3}{\csname#1\endcsname#3}}
\newcommand{\foperator}[1]{\mathop{{#1}\empty{}}}
\newcommand{\f}[3][]{\ensuremath{\foperator{#2}\parentheses[#1]{(}{)}{#3}}}

\newcommand{\iverson}[1]{\ensuremath{\parentheses{[}{]}{#1}}}
\WithSuffix\newcommand{\iverson}*[1]{\ensuremath{\iverson{\text{\normalfont #1}}}}

\newcommand{\Ohsymbol}{O}
\newcommand{\Oh}[2][]{\ensuremath{\f[#1]{\Ohsymbol}{#2}}}
\newcommand{\set}[1]{\ensuremath{\parentheses{\{}{\}}{#1}}}
\WithSuffix\newcommand{\set}*[2]{\ensuremath{%
    \left\{{#1}\thinspace\setmiddlesymbol\thinspace{#2}\right\}}}

\newcommand{\dd}{\mathrm{d}}

\newcommand{\calC}{\mathcal{C}}
\newcommand{\calP}{\mathcal{P}}
\newcommand{\calS}{\mathcal{S}}
\newcommand{\calM}{\mathcal{M}}
\newcommand{\calL}{\mathcal{L}}

\renewcommand{\P}[2][]{\f[#1]{\mathbb{P}}{#2}}

\newcommand{\N}{\ensuremath{\mathbbm{N}}}

\newcommand{\TODO}[1]%
{\par\fbox{\begin{minipage}{0.9\linewidth}\textbf{TODO:} #1\end{minipage}}\par}

\newtheorem{thms}{Thm}[section] 
\theoremstyle{plain}
\newtheorem*{theorem*}{Theorem}
\newtheorem{theorem}[thms]{Theorem}
\newtheorem{lemma}[thms]{Lemma}
\newtheorem{corollary}[thms]{Corollary}
\newtheorem{proposition}[thms]{Proposition}
\theoremstyle{definition}

\theoremstyle{remark}
\newtheorem{remark}[thms]{Remark}

\newcommand{\Halt}{H^{\mathrm{alt}}}
\newcommand{\Li}[3][]{\ensuremath{\f[#1]{\operatorname{Li}_{#2}}{#3}}}
\DeclareMathOperator{\artanh}{artanh}

\newcommand{\cv}{\mathrm{cv}}
\newcommand{\ct}{\mathrm{ct}}
\newcommand{\spf}{\mathrm{sf}}

\numberwithin{equation}{section}

\title{An Extended Note on the Comparison-optimal Dual-Pivot Quickselect}

\author{Daniel Krenn}
\address{Daniel Krenn,
  Institut f\"ur Mathematik, Alpen-Adria-Universit\"at Klagenfurt,
  Universit\"atsstra\ss e 65--67, 9020 Klagenfurt am W\"orthersee, Austria}
\email{\href{mailto:math@danielkrenn.at}{math@danielkrenn.at} \textit{or}
  \href{mailto:daniel.krenn@aau.at}{daniel.krenn@aau.at}}

\thanks{The author is supported by the
   Austrian Science Fund (FWF): P\,24644-N26.}

 \thanks{The author kindly thanks Helmut Prodinger for his inspiring
   talk ``Quickselect, multiple Quickselect, Quicksort with
   median-of-three partition and related material'' given at AAU
   Klagenfurt in May 2016.}


\begin{document}
\maketitle


\begin{abstract}
  In this note the precise minimum number of key comparisons any
  dual-pivot quickselect algorithm (without sampling) needs on average
  is determined. The result is in the form of exact as well as
  asymptotic formul\ae{} of this number of a comparison-optimal
  algorithm. It turns out that the main terms of these asymptotic
  expansions coincide with the main terms of the corresponding
  analysis of the classical quickselect, but still---as this was shown
  for Yaroslavskiy quickselect---more comparisons are needed in the
  dual-pivot variant. The results are obtained by solving a second
  order differential equation for the generating
  function obtained from a recursive approach.
\end{abstract}


\section{Introduction}
\label{sec:intro}


Quickselect~\cite{Hoare:1961:find} (also called ``Hoare's find
algorithm'' or ``Hoare's selection algorithm'') is an algorithm to
select the $j$th smallest element (the ``$j$th rank'') of an unordered list. It uses the
same partitioning strategy as quicksort~\cite{Hoare:1961:quicksort,
Hoare62, Knuth:1998:Art:3}: One element of the list is chosen as a pivot
element and the remaining are split into two sublists containing the
elements smaller and larger than the pivot.
Both algorithms then proceed recursively on the sublists (quicksort)
or on one sublist (quickselect).

\subsection{Quicksort}
\label{sec:quicksort}

The classical quicksort algorithm with one pivot element needs $2n\log
n + \Oh{n}$, as $n\to\infty$, key comparisons on average to sort a
list of length~$n$.  Using more than one pivot element can decrease
this number. For example,
Yaroslavskiy's~\cite{Yaroslavskiy-Mailinglist} partitioning strategy
and dual-pivot quicksort algorithm results in only $1.9n\log n +
\Oh{n}$, see Wild and Nebel~\cite{Wild-Nebel:2012:avg-quicksort}. This
can be improved further. The lower bound for dual-pivot quicksort is
$1.8n\log n + \Oh{n}$ key comparisons; this was shown in Aumüller and
Dietzfelbinger~\cite{Aumueller-Dietzfelbinger:ta:optim-partit}.  Their
optimal/minimal strategy called ``Clairvoyant'' uses an oracle, and therefore
it is non-algorithmic. Its algorithmic version ``Count'' still only
needs $1.8n\log n + \Oh{n}$ key comparisons. The precise analysis of
\cite{Aumueller-Diezfelbinger-Heuberger-Krenn-Prodinger:2016:quicksort-paths-arxiv}
reveals the linear terms of these two strategies, and it is claimed
that ``Count'' is the optimal partitioning strategy.

Note that all strategies considered in this article choose the pivots
without sampling.
A survey on quicksort with a special focus on dual-pivot partitioning can be found in Wild~\cite{Wild:2016:phd}.


\subsection{Single-Pivot vs.\@ Dual-Pivot Quickselect}
\label{sec:quickselect}


We use $H_n = \sum_{k=1}^n 1/k$ to denote the harmonic numbers.

Due to the improvements of quicksort with dual-pivoting which were
mentioned above, one would expect that a dual-pivot quickselect needs
as well fewer key comparisons than the classical quickselect. However,
Wild, Nebel and Mahmoud~\cite{Wild-Nebel-Mahmoud:2016:quickselect}
show that this is not true. While the classical quickselect needs
\begin{equation}\label{eq:average-classical}
  3n - 8 H_n + 13 - 8 n^{-1} H_n
  = 3n - 8 \log n - 8\gamma + 13 + \Oh{n^{-1}\log n}
\end{equation}
key comparisons on average when selecting a rank chosen uniformly at
random, see Mahmoud, Modarres and
Smythe~\cite{Mahmoud-Modarres-Smythe:1995:quickselect}, quickselect
with Yaroslavskiy's partitioning
strategy~\cite{Wild-Nebel-Mahmoud:2016:quickselect} needs
\begin{multline}\label{eq:average-Yaroslavskiy}
  \tfrac{19}{6}n - \tfrac{37}{5}H_n + \tfrac{1183}{100}
  - \tfrac{37}{5} n^{-1} H_n - \tfrac{71}{300} n^{-1} \\
  = \tfrac{19}{6}n - \tfrac{37}{5}\log n
  - \tfrac{37}{5}\gamma + \tfrac{1183}{100} +\Oh{n^{-1}\log n}
\end{multline}
key comparisons. The same is true for the average
number of key comparisons when selecting the smallest or largest
rank. There it increases from
\begin{equation}\label{eq:smallest-classical}
  2n - 2 H_n
  = 2n - 2\log n - 2\gamma + \Oh{n^{-1}}
\end{equation}
of the classical
quickselect~\cite{Mahmoud-Modarres-Smythe:1995:quickselect} to
\begin{multline}\label{eq:smallest-Yaroslavskiy}
  \frac{57n^4 - 48n^3 H_n - 178n^3 + 144n^2 H_n + 135n^2 - 96n H_n - 14n + 24}{
    24n(n-1)(n-2)} \\
  = \tfrac{19}{8}n - 2\log n - 2\gamma - \tfrac{7}{24} +\Oh{n^{-1}}
\end{multline}
of Yaroslavskiy's
quickselect~\cite{Wild-Nebel-Mahmoud:2016:quickselect}.
The latter reference, as well as~\cite{Wild:2016:phd}, provide
further discussions and insights.

The question that is answered in this note is: Does any dual-pivot
quickselect with the comparison-optimal partitioning strategy beat (in
terms of the number of key comparisons) the classical quickselect or
not?


\subsection{Discussion: The New Results Face to Face with the Existing Results}
\label{sec:results}


The aim of this note is to determine a lower bound for all dual-pivot
quickselect algorithms by counting the number of key comparisons
in quickselect using the optimal paritioning strategy ``Count''.

On the one hand, we analyze selecting a random rank (``grand averages'').
This results in
\begin{equation}\label{eq:average-j-min}
  \overline C^{\min}_n =
  3 n
  + \frac{3}{20} (\log n)^2
  + \left(\frac{\gamma + \log 2}{10} + \frac{319}{50}\right) \log n
  + \Oh{1}
\end{equation}
key comparisons on average (expected value), formulated precisely as
Theorem~\ref{thm:average-j-exact} and
Corollary~\ref{cor:average-j-asy}. As expected, this number of key
comparisons is (asymptotically) lower than the number in Yaroslavskiy
quickselect~\eqref{eq:average-Yaroslavskiy} which has main term
$\tfrac{19}{6}n$. We even get the same main term $3n$ as in the
classical quickselect~\eqref{eq:average-classical}. Unfortunately the
second order term in \eqref{eq:average-j-min} is still larger than the
second order term in~\eqref{eq:average-classical}. Thus, we can answer
the question posed above, whether a dual-pivot quickselect beats the
classical quickselect, by ``no''---at least when selecting a random
rank.

On the other hand, we analyze selecting the $j$th smallest/largest rank
with $j\in\set{1,2,3,4}$ which results in
\begin{equation}\label{eq:smallest-min}
  C^{\min}_{n,j} = C^{\min}_{n,n-j+1} = 
  \frac{9}{4} n + \frac{1}{12} (\log n)^2
  +  \left(\frac{\gamma+\log 2}{6} + t_j\right) \log n
  + \Oh{1}
\end{equation}
key comparisons on average. There the $t_j$ are explicitly known
constants. See Section~\ref{sec:fixed-j} for details.  Again the main
term is lower than that of the Yaroslavski
variant~\eqref{eq:smallest-Yaroslavskiy}, but it is still larger than
the main term of the classical quickselect~\eqref{eq:smallest-classical}.
So again our main question is answered by a ``no''.

We also analyze the theoretical (non-algorithmic) ``Clairvoyant''
partitioning strategy,
see~\cite{Aumueller-Dietzfelbinger:ta:optim-partit,
  Aumueller-Diezfelbinger-Heuberger-Krenn-Prodinger:2016:quicksort-paths-arxiv}
and Section~\ref{sec:part}. It turns out that the main term of the
average number of key comparisons is the same as
in~\eqref{eq:average-j-min} and~\eqref{eq:gf-cost-count} respectively,
but surprisingly its second order term has the opposite sign. Thus it
needs fewer key comparisons than the classical quickselect
(formul\ae{}~\eqref{eq:average-classical}
and~\eqref{eq:smallest-classical}).
Details are to be found at the end of Sections~\ref{sec:average-j}
and~\ref{sec:fixed-j}.


\subsection{What Else?}
\label{sec:else}


Many other properties and variants of the (classical) quickselect are
studied and can be extended to dual-pivot
quickselect algorithms and can be investigated for them.
Prodinger~\cite{Prodinger:1995:quickselect},
Lent and Mahmoud~\cite{Lent-Mahmoud:1996:multiple-quickselect},
Panholzer and Prodinger~\cite{Panholzer-Prodinger:1998:gf-quickselect},
and Kuba~\cite{Kuba:2006:quickselect}
analyze quickselect when selecting multiple ranks simultaneously.
Different strategies to choose the pivot are possible as well. For example,
Kirschenhofer, Prodinger and Martinez~\cite{Kirschenhofer-Prodinger-Martinez:1997:hoare-find-median-of-3} use a median of three strategy.

Distributional results and higher moments such as the variance are also 
feasible. For Yaroslavskiy's
quicksort, this was done by Wild, Nebel and Neininger~\cite{WildNN15}
and for the corresponding quickselect by Wild, Nebel and
Mahmoud~\cite{Wild-Nebel-Mahmoud:2016:quickselect}. It is possible 
to extend the methods of the
latter for our optimal
paritioning strategy; this is a task for the full version of
this extended abstract.


\subsection{Notation: Harmonic Numbers and More}
\label{sec:notation}


Here a short note on the notation used in the sections below.
There are
\begin{itemize}
\item the harmonic numbers $H_n = \sum_{k=1}^n 1/k$ and
\item the alternating harmonic numbers $\Halt_n = \sum_{k=1}^n (-1)^k/k$.
\end{itemize}
Moreover, we use
\begin{itemize}
\item the Iversonian notation
  \begin{equation*}
    \iverson{\mathit{expr}} =
    \begin{cases}
      1&\text{ if $\mathit{expr}$ is true},\\
      0&\text{ if $\mathit{expr}$ is false},
    \end{cases}
  \end{equation*}
  which was popularized by Graham, Knuth, and
  Patashnik~\cite{Graham-Knuth-Patashnik:1994}.
\end{itemize}
By $\gamma = 0.5772156649\dots$, we denote the Euler--Mascheroni constant.


\section{Partitioning Strategies}
\label{sec:part}


As mentioned in the introduction, the average number of comparisons
for a dual-pivot quicksort or quickselect algorithm depends on its
partitioning strategy. So let us suppose we have an (unsorted) list
of distinct elements. We choose the first and the last element
as pivot elements $p$ and $q$. We assume $p<q$; this needs one
comparison.

Informally, a partitioning strategy is an algorithm, which, in each step,
\begin{enumerate}
\item takes an unclassified element,
\item compares it with $p$ or $q$ first,
\item if not already classified compares it with the remaining element
  $p$ or $q$, and
\item marks the element as small ($<p$), medium (between $p$ and $q$) or large ($>q$).
\end{enumerate}
The choice whether to choose $p$ or $q$ for the first comparison in
each step may depend on the history of the outcome of the previous
classifications. Additionally the index of the element to read may
depend on this history as well. However, the index of the element to
read does not have any influence on the results presented in this
article.

A more formal definition of partitioning strategies can be found in
Aumüller and
Dietzfelbinger~\cite{Aumueller-Dietzfelbinger:ta:optim-partit}; they
use the following decision trees to model a partitioning strategy: A
strategy is described by a complete rooted ternary tree with $n-2$
levels (as $n-2$ elements have to be classified). Each vertex is
labeled by a pair consisting of the index of the element to be
classified and of $p$ or $q$ indicating which element to use for the
first comparison for the classification. The three outgoing edges of a
vertex are labeled by small, medium and large, respectively, and
represent the outcome of the classification. Every order/permutation
of a list of elements corresponds to a path in this tree which starts
at the root and ends in a leaf.

Next, we describe a couple of partitioning strategies.
\begin{description}
\item [``Smaller pivot first''] We always compare with the smaller
  pivot first. Each small element needs only one comparison to be
  classified, each medium and each large element needs two
  comparisons. This results in
  \begin{equation*}
    P^\spf_n = \frac{5}{3}n - \frac{7}{3}
  \end{equation*}
  for the
  expected number of key comparisions to classify a list of $n\geq2$
  elements. (Two of these list-elements will be the pivots.) 
  The corresponding generating function of the expected 
  cost of partitioning is
  \begin{equation*}
    \f{P^\spf}{z} = \frac{5}{3(1-z)^2} - \frac{4}{1-z} - \frac{2}{3}(1-z) + 3.
  \end{equation*}
  See also Appendix~\ref{sec:p-first} for details. Note that the very
  same result holds for the ``larger pivot first'' partitioning
  strategy by symmetry.

\item [``Yaroslavskiy'' (\cite{Yaroslavskiy-Mailinglist})]
  See the introduction for details and references.

\item [``Count''] We keep track of the numbers of already classified
  small and large elements. If there were more larger than smaller
  elements up to now, then we use $q$ for the first comparison
  in the next step, otherwise $p$.

  This is the optimal---meaning that it minimizes the expected number
  of key comparisons---algorithmic dual-pivot partitioning strategy,
  see \cite{Aumueller-Diezfelbinger-Heuberger-Krenn-Prodinger:2016:quicksort-paths-arxiv}. The expected number of key comparisons to classify
  a list of $n$ elements (two of these elements will be the pivots) is
  \begin{equation*}
    P^\ct_n = 
    \frac{3}{2}n + \frac{1}{4} \log n + \frac{2\gamma+2\log 2-19}{8}
    + \Oh{n^{-1}}.
  \end{equation*}
  It was analyzed in \cite{Aumueller-Diezfelbinger-Heuberger-Krenn-Prodinger:2016:quicksort-paths-arxiv}, where an exact formula and
  a precise asymptotic expansion was stated.
  The corresponding generating function of the expected 
  cost of partitioning is known explicitly as
  \begin{equation}\label{eq:gf-cost-count}
    \f{P^\ct}{z} =
    \frac{3}{2(1-z)^2} + \frac{\artanh(z)}{2 {(1-z)}}
    -\frac{31z^2}{8(1-z)}-\frac{3+z}{8}\artanh(z) - \frac{3}{2}-\frac{25z}{8}
  \end{equation}
  from \cite{Aumueller-Diezfelbinger-Heuberger-Krenn-Prodinger:2016:quicksort-paths-arxiv} as well.

  This article's main focus is on the partitioning strategy ``Count''.

\item [``Clairvoyant''] This strategy uses an oracle to predict the
  number of small and large elements in the remaining (unsorted) list. 
  If there are going to be more larger than smaller
  elements, then we use $q$ for the first comparison, otherwise $p$.

  Note that this strategy is not algorithmic. It provides a theoretic
  lower bound for the number of key comparisons of all partitioning
  strategies~\cite{Aumueller-Dietzfelbinger:ta:optim-partit}. Again,
  an explicit analysis can be found in
  \cite{Aumueller-Dietzfelbinger:ta:optim-partit} and
  \cite{Aumueller-Diezfelbinger-Heuberger-Krenn-Prodinger:2016:quicksort-paths-arxiv}.
  The expected number of key comparisons to classify a list of $n$
  elements (two of these elements will be the pivots) is
  \begin{equation*}
    P^\cv_n = 
    \frac{3}{2}n - \frac{1}{4} \log n - \frac{2\gamma+2\log 2+13}{8}
    + \Oh{n^{-1}}.
  \end{equation*}
  
\end{description}

When using these strategies for quickselect,
randomness in the obtained sublists after the
partitioning step is preserved. We refer here to Wild, Nebel and
Mahmoud~\cite{Wild-Nebel-Mahmoud:2016:quickselect}, who use a criterion
of Hennequin~\cite{Hennequin:1989:quicksort}. See also the third volume
of the book of Knuth~\cite{Knuth:1998:Art:3}.


\section{The Recurrence}
\label{sec:recurrence}

Let $n\in\N_0$. We assume that the input of our quickselect algorithm
is a random permutation of $\set{1,\dots,n}$ chosen uniformly at
random. For $j\in\set{1,\dots,n}$, let us denote by $C_{n,j}$ the
average number of comparisons needed to select the $j$th smallest
element.

By symmetry of the algorithm, selecting the $j$th largest element
costs as much as selecting the $j$th smallest element, thus we have
\begin{equation}\label{eq:C-symmetry}
  C_{n,j} = C_{n,n-j+1}.
\end{equation}

The average number of comparisons satisfies the following recurrence.

\begin{proposition}\label{pro:recurrence}
  Let $j\in\set{1,\dots,n}$. Then
  \begin{equation*}
    C_{n,j} = P_n + S_{n,j} + M_{n,j} + L_{n,j}
  \end{equation*}
  with
  \begin{align*}
    S_{n,j} &= \frac{1}{\binom{n}{2}} \sum_{s=j}^{n-2} (n-1-s) C_{s,j}, \\
    M_{n,j} &= \frac{1}{\binom{n}{2}} \sum_{m=1}^{n-2}
    \sum_{s=\max\set{0,j-m-1}}^{\min\set{j-2, n-m-2}} C_{m,j-s-1}, \\
    L_{n,j} &= \frac{1}{\binom{n}{2}}
    \sum_{\ell=n-j+1}^{n-2} (n-1-\ell) C_{\ell,n-j+1},
  \end{align*}
  for $n\geq2$, and $C_{0,j}=0$ and $C_{1,j}=0$.
\end{proposition}

The special case of the recurrence for $j=1$ can be found in
\cite{Wild-Nebel-Mahmoud:2016:quickselect}. There, a recurrence for
analyzing the grand averages is presented as well.

\begin{proof}[Proof of Proposition~\ref{pro:recurrence}]
  We assume that the input is a random permutation of
  $\set{1,\dots,n}$. The expected cost $C_{n,j}$ is the sum of the
  expected partitioning cost $P_n$ and the sum of the cost of the
  recursive call for the small elements $S_{n,j}$, medium
  elements~$M_{n,j}$ or large elements~$L_{n,j}$. Throughout this
  proof, the random variables of the number of small, medium and large
  elements are denoted by $S$, $M$ and $L$, respectively, and we
  have $n-2 = S+M+L$.

  After the partitioning step, we proceed with the small elements if
  the number $S$ of small elements is at least $j$; this number can be
  at most $n-2$ because of the two pivots $p$ and $q$. For a fixed
  realization
  $S$, there are $n-1-S$ possibilities---all of them are equally
  likely---to partition the medium and large elements. This results in
  the probability $\P{S=s}=(n-1-s)/\binom{n}{2}$ to continue with selecting
  the $j$th smallest element of a list of $s$ elements; the expected
  cost for this is $C_{s,j}$. The quantity $S_{n,j}$ follows by summing up
  over all $s$.

  Similarly, the number $L$ of large elements has to be at least $n-j+1$
  to recurs into the large-branch. There are $n-1-L$ possibilities, thus
  $\P{L=\ell}=(n-1-\ell)/\binom{n}{2}$ for every $\ell$. For a fixed $\ell$,
  we need to find the $(j-n+\ell)$th smallest element
  (as $n-\ell=s+m+2$), so the cost is $C_{\ell,j-n+\ell} =
  C_{\ell,n-j+1}$ by symmetry~\eqref{eq:C-symmetry}. The result for $L_{n,j}$
  follows.

  In order to recurs on the medium elements, we need $S$ to be at most
  $j-2$ and $L$ to be at most $n-j-1$; both have $0$ as a lower bound.
  All events are equally likely which results in the probability
  $\P{S=s, L=\ell}=1/\binom{n}{2}$. The expected cost is
  $C_{n-2-s-\ell,j-s-1}$ as $m=n-2-\ell$ and we continue to find the
  $(j-s-1)$st element. Summing up and rewriting the resulting double sum
  in terms of the indices $s$ and $m$ (instead of $s$ and $\ell$)
  yields $M_{n,j}$. This completes the proof.
\end{proof}

We translate the recurrence above into the world of generating
functions.  We set $\f{C}{z,u} = \sum_{n,j} C_{n,j} z^n u^j$, and, for
the number of comparisons for partitioning, we define $\f{P}{z} =
\sum_{n} P_n z^n$.

The symmetry~\eqref{eq:C-symmetry} translates to the functional
equation
\begin{equation}\label{eq:C-functional-eq}
  u \f{C}{zu,1/u}
  = \sum_{n,j} C_{n,j} z^n u^{n-j+1}
  = \sum_{n,j} C_{n,n-j+1} z^n u^j
  = \f{C}{z,u}.
\end{equation}
We need this functional equation in the proof below. The generating
function obtained by the recurrence of
Proposition~\ref{pro:recurrence} satisfies the following ordinary
differential equation in the variable~$z$.

\begin{proposition}\label{pro:C-ode}
  We have
  \begin{equation*}
    \frac{\partial^2}{\dd z^2}\f{C}{z,u}
    = \frac{u}{1-u} \left(\f{P''}{z} - u^2 \f{P''}{zu}\right)
    + 2 \f{C}{z,u} \f{r}{z,u}
  \end{equation*}
  with
  \begin{equation*}
    \f{r}{z,u} = 
    \frac{1}{(1-z)^2} + \frac{u}{(1-z)(1-zu)} + \frac{u^2}{(1-zu)^2}.
  \end{equation*}
  If $u=1$, then we have
  \begin{equation*}
    \frac{\partial^2}{\dd z^2}\f{C}{z,u}\Bigr\vert_{u=1}
    = \frac{1}{z} \bigl(z^2 \f{P''}{z}\bigr)'
    + \frac{6}{(1-z)^2} \f{C}{z,1}.
  \end{equation*}
\end{proposition}

Note that a generating function and an ordinary differential equation
for the grand averages---this is
equivalent to considering~$\f{C}{z,1}$---for the particular
Yaroslavskiy quickselect can be found
in~\cite{Wild-Nebel-Mahmoud:2016:quickselect}.

The full proof of Proposition~\ref{pro:C-ode} 
can be found in Appendix~\ref{sec:appendix-recurrence}.

\begin{proof}[Sketch of the proof of Proposition~\ref{pro:C-ode}]
  We use the recurrence of Proposition~\ref{pro:recurrence} to obtain
  \begin{multline*}
    n(n-1) C_{n,j} = n(n-1) P_n \iverson{1\leq j \leq n} \\
    + 2 \sum_{s=0}^{n-1} (n-1-s) C_{s,j}
    + 2 \sum_{m=0}^{n-2} \sum_{s=0}^{n-m-2} C_{m,j-s-1}
    + 2 \sum_{\ell=0}^{n-1} (n-1-\ell) C_{\ell,n-j+1}.
  \end{multline*}
  We multiply by $z^{n-2} u^j$ and sum up over
  all $n\geq2$ and all $j$; we treat each summand separately, so we
  have an equation $\calC=\calP+\calS+\calM+\calL$.

  The parts $\calC$ and $\calP$ are straight forward to determine.

  Next, we deal with $\calS$.
  We extend the sum by including $n=1$, then shift from $n-1$ to $n$, and get
  \begin{align*}
    \calS
    &= 2 \sum_{j} \sum_{n\geq2} \sum_{s=0}^{n-1} (n-1-s) C_{s,j} z^{n-2} u^j \\
    &= 2 \sum_{j} \sum_{n\geq1} \sum_{s=0}^{n-1} (n-1-s) C_{s,j} z^{n-2} u^j \\
    &= 2 \sum_{j} \sum_{n\geq0} \sum_{s=0}^n (n-s) z^{n-s-1} C_{s,j} z^s u^j.
  \end{align*}
  Rewriting the convolution to a product of generating functions yields
  \begin{equation*}
    \calS
    = 2 \Bigl( \sum_{n\geq0} n z^{n-1} \Bigr)
    \sum_{j} \sum_{n\geq0} C_{n,j} z^n u^j
    = 2 \Bigl(\frac{1}{1-z}\Bigr)' \f{C}{z,u}
    = \frac{2}{(1-z)^2} \f{C}{z,u}.
  \end{equation*}

  We proceed in a similar manner with $\calL$,
  where \eqref{eq:C-functional-eq} has to be used.
  To deal with the sum $\calM$, we have to take into account one
  additional summation; we succeed by proceeding as above.
  The overall result follows as $\calC=\calP+\calS+\calM+\calL$.
\end{proof}


\section{A Random Selection}
\label{sec:average-j}


We focus on the partitioning strategy ``Count'', see Section~\ref{sec:part} for details, which minimizes the
number of key comparisons among all dual-pivot partitioning strategies.

Let $n\in\N_0$ be fixed. In this section, we assume that $j$ is an
integer of $\set{1,\dots,n}$ chosen uniformly at random. This means
for our algorithm, that we perform a random selection.  The input is
again a random permutation of $\set{1,\dots,n}$. We study the expected value/average
number $\overline C^\ct_n$ of key comparisons of this selection depending on
the input size~$n$; the following theorem holds.

\begin{theorem}\label{thm:average-j-exact}
  The average number (expected value) of key comparisons in the comparison-optimal
  dual-pivot quickselect algorithm---it uses strategy ``Count''---when
  performing a random selection is
  \begin{multline*}
    \overline C^\ct_n =
3 n
+\frac{3}{20n} \sum_{k=1}^{n-1} H_k H_{n-k}
-\frac{3}{10n} \sum_{k=1}^n \frac{\Halt_{k-1}}{k} (n-k+1)
-\frac{194}{25} H_n
+\frac{9}{25} \Halt_n
+\frac{1564}{125}
\\
-\frac{1527}{200} \frac{H_n}{n}
+\frac{47}{200} \frac{\Halt_n}{n}
+\frac{783}{4000n}
-\frac{9}{50} \frac{(-1)^n}{n}
\\
+\frac{22}{1600n}\left(
  \frac{n-1}{n(n-2)}\iverson*{$n$ odd} -
  \frac{n-5}{(n-1)(n-3)}\iverson*{$n$ even}
\right)
  \end{multline*}
  for $n\geq4$.
\end{theorem}

We have $\overline C_0=\overline C_1=0$, $\overline C_2=8/3$ and
$\overline C_3=9/2$.
We extract the asymptotic behavior out of the
generating function used in the proof of
Theorem~\ref{thm:average-j-exact}; this is the corollary below.

\begin{corollary}\label{cor:average-j-asy}
  The average number (expected value) of key comparisons in the comparison-optimal
  dual-pivot quickselect algorithm---it uses strategy ``Count''---when
  performing a random selection is
  \begin{equation*}
    \overline C^\ct_n = 
    3 n
    + \frac{3}{20} (\log n)^2
    + \left(\frac{\gamma + \log 2}{10} + \frac{319}{50}\right) \log n
    + \Oh{1}
  \end{equation*}
  asymptotically as $n$ tends to infinity.
\end{corollary}


\begin{proof}[Proof of Theorem~\ref{thm:average-j-exact}
  and Corollary~\ref{cor:average-j-asy}]
  Proposition~\ref{pro:C-ode} provides an ordinary differential
  equation for $\f{C}{z,1}$.
  As this linear differential equation is
  basically the same---it only differs in the inhomogeneity---as for
  the dual-pivot quicksort, its solution is
  \begin{equation}\label{eq:sol-qsort-ode}
    \f{C}{z,1} = (1-z)^3\int_{0}^z (1-t)^{-6}\int_{0}^t (1-s)^3
    \frac{1}{s} \bigl(s^2 \f{P''}{s}\bigr)'
    \,ds\,dt
  \end{equation}
  as described in Wild~\cite{Wild2013} (who follows Hennequin~\cite{Hennequin:1991:analy};
  see also \cite{Aumueller-Diezfelbinger-Heuberger-Krenn-Prodinger:2016:quicksort-paths-arxiv}
  for the explicit solution).

  We use $\f{P}{z} = \f{P^\ct}{z}$ (and write $\f{C^\ct}{z,1}$
  instead of $\f{C}{z,1}$).
  By performing the integration~\eqref{eq:sol-qsort-ode}, we obtain
  the generating function
  \begin{multline*}
    \f{C^\ct}{z,1} =
\frac{6}{{(1-z)}^{3}}
+ \frac{3 \, \log(1-z)^{2}}{20 \, {(1-z)}^{2}}
- \frac{3}{10(1-z)^2} \f{L_2}{z}
+ \frac{194 \, \log(1-z)}{25 \, {(1-z)}^{2}}
\\
- \frac{9 \, \log(1+z)}{25 \, {(1-z)}^{2}}
- \frac{531}{125 \, {(1-z)}^{2}}
+ \frac{\log(1+z)}{8 \, {(1-z)}}
- \frac{\log(1-z)}{8 \, {(1-z)}}
- \frac{1389}{800 \, {(1-z)}}
\\
- \frac{11}{3200} \, {(1-z)}^{3} \log(1-z)
+ \frac{11}{3200} \, {(1-z)}^{3} \log(1+z)
\\
- \frac{29}{750} \, {(1-z)}^{3}
+ \frac{11}{1600} \, {(1-z)}^{2}
- \frac{11}{1600} \, z
+ \frac{77}{4800}.
  \end{multline*}
  Here we use the abbreviation
  \begin{equation*}
    \f{L_2}{z} = -\int_0^z \frac{\log (1+t)}{1-t} \dd t,
  \end{equation*}
  see Appendix~\ref{sec:preparation}.
  Theorem~\ref{thm:average-j-exact} follows by extracting the
  coefficients of the generating function exactly, whereas
  Corollary~\ref{cor:average-j-asy} follows by extracting the
  coefficients asymptotically via singularity
  analysis~\cite{Flajolet-Odlyzko:1990:singul,
    Flajolet-Sedgewick:ta:analy}. Appendix~\ref{sec:preparation} might
  assist.
\end{proof}


The authors of \cite{Aumueller-Dietzfelbinger:ta:optim-partit} and
\cite{Aumueller-Diezfelbinger-Heuberger-Krenn-Prodinger:2016:quicksort-paths-arxiv}
study the partitioning strategy ``Clairvoyant'' which is based on an
oracle, see Section~\ref{sec:part} for details. Our methods here can
be easily modified to obtain results for this strategy as well.


\begin{theorem}\label{thm:average-j-clairvoyant}
  The average number (expected value) of key comparisons in the
  dual-pivot quickselect algorithm with strategy ``Clairvoyant'' when
  performing a random selection is
  \begin{multline*}
    \overline C^\cv_n = 
  3 n
  -\frac{3}{20} \sum_{k=1}^{n-1} H_k H_{n-k}
  + \frac{3}{10} \frac{1}{n} \sum_{k=1}^n \frac{\Halt_{k-1}}{k} (n-k+1)
  -\frac{196}{25} H_n
  -\frac{9}{25} \Halt_n
  + \frac{1576}{125} \\
  -\frac{1593}{200} \frac{H_n}{n}
  -\frac{47}{200} \frac{\Halt_n}{n}
  -\frac{703}{4000} \frac{1}{n} 
  + \frac{9}{50} \frac{(-1)^n}{n} \\
  + \frac{22}{1600} \frac{1}{n}
  \left( \frac{n-1}{n(n-2)} \iverson*{$n$ odd}
    - \frac{n-5}{(n-1)(n-3)} \iverson*{$n$ even} \right).
  \end{multline*}
  This equals
  \begin{equation*}
    \overline C^\cv_n = 
  3 n - \frac{3}{20} (\log n)^2
  + \left(-\frac{3\gamma + 3\log 2}{10} + \frac{461}{50}\right) \log n
  + \Oh{1}
  \end{equation*}
  asymptotically as $n$ tends to infinity.
\end{theorem}

The proof of Theorem~\ref{thm:average-j-clairvoyant} can be found in
Appendix~\ref{sec:appendix-proofs}.


For completeness, we include the expected value/average number of key comparisons for
dual-pivot quickselect with the partitioning strategy ``smaller pivot
first'' here. Note that these results are equal to those of the
strategy ``larger pivot first'' by symmetry.

\begin{proposition}\label{pro:average-j-spf}
  The average number (expected value) of key comparisons in the dual-pivot quickselect
  algorithm with strategy ``smaller pivot first'' when performing a
  random selection is
  \begin{equation*}
    \overline C^\spf_n = 
    \frac{10}{3} n
    -\frac{44}{5} H_n
    + \frac{354}{25}
    -\frac{44}{5} \frac{H_n}{n}
    + \frac{2}{75}.
  \end{equation*}
  This equals
  \begin{equation*}
    \overline C^\spf_n = 
    \frac{10}{3} n + \frac{44}{5} \log n + \frac{44}{5}\gamma
    - \frac{758}{75} + \frac{12}{5} n^{-1} + \Oh{n^{-2}}
  \end{equation*}
  asymptotically as $n$ tends to infinity.
\end{proposition}


\section{Selecting the $j$th Smallest/Largest Element}
\label{sec:fixed-j}


In this section, we determine the expected value/average number of key comparisons for
selecting, among others, the smallest ($j=1$) or largest element
($j=n$) of a random permutation of $\set{1,\dots,n}$, all equally
likely. Again we use the partitioning strategy ``Count''
(Section~\ref{sec:part}).

We use the bivariate generating function $\f{C}{z,u}$ of
Section~\ref{sec:recurrence}. Let $j\in\set{1,\dots,n}$, and let us
group $\f{C}{z,u}$ in terms of the parameter $j$ as
\begin{equation*}
  \f{C}{z,u} = \sum_{j\geq1} \f{C_j}{z} u^j.
\end{equation*}
We extract the $j$th coefficient of the differential equation for
$\f{C}{z,u}$ of Proposition~\ref{pro:C-ode}. This leads to the
following system of ordinary differential equations.  Note that
$\f{C_1}{z}$ in the case of Yaroslavskiy quickselect is stated
in~\cite{Wild-Nebel-Mahmoud:2016:quickselect}.

\begin{lemma}\label{lem:ode-system}
  We have
  \begin{equation*}
    \f{C_j''}{z} - \frac{2}{(1-z)^2} \f{C_j}{z}
    = \f{Q_j}{z}
  \end{equation*}
  with
  \begin{equation*}
    \f{Q_j}{z} = \f{P''}{z} - \sum_{n<j} n(n-1) P_n z^{n-2}
    + 2 \sum_{k=0}^{j-1} \f{C_k}{z} z^{j-k-2} \left(\frac{z}{1-z}+j-k-1\right)
  \end{equation*}
  and $\f{C_j}{0}=\f{C_j'}{0}=0$.
\end{lemma}

The proof is straight forward and can be found in
Appendix~\ref{sec:appendix-proofs}.


\begin{remark}\label{rem:ode-qselect-system-solution}
  The ordinary differential equation
  \begin{equation*}
    \f{C''}{z} - \frac{2}{(1-z)^2} \f{C}{z} = \f{Q}{z}
  \end{equation*}
  with $\f{C}{0}=\f{C'}{0}=0$ has the solution
  \begin{equation}\label{eq:ode-qselect-solution}
    \f{C}{z} = (1-z)^2 \int_0^z (1-t)^{-4} \int_0^t (1-s)^2 \f{Q}{s} \dd s \dd t.
  \end{equation}
  This provides a way to solve for $\f{C_j}{z}$ of
  Lemma~\ref{lem:ode-system}.
\end{remark}

The proof of Remark~\ref{rem:ode-qselect-system-solution} can be found
in Appendix~\ref{sec:appendix-proofs}.

We are now able to obtain cost coefficients as stated in the following
proposition.

\begin{proposition}\label{pro:smallest-exact}
  The average number (expected value) of key comparisons in the comparison-optimal
  dual-pivot quickselect algorithm---it uses strategy ``Count''---when
  selecting the smallest or largest element is
  \begin{multline*}
    C^\ct_{n,1} = C^\ct_{n,n} =
    \frac{9}{4} n
    +\frac{1}{12} \sum_{k=1}^{n-1} \frac{H_k}{n-k}
    -\frac{1}{6} \sum_{k=2}^n \frac{\Halt_{k-1}}{k}
    -\frac{43}{18} H_n
    +\frac{1}{18} \Halt_n \\
    +\frac{5}{108}
    +\frac{\iverson*{\text{$n$ odd}}(n-1)}{36n(n-2)} 
    -\frac{\iverson*{\text{$n$ even}}}{36(n-1)}.
  \end{multline*}
\end{proposition}

Note that one can rewrite this exact formula, in particular
$\sum_{k=1}^{n-1} H_k/(n-k)$, in terms of other variants of the
harmonic numbers, see \cite{Greene-Knuth:1990:mathem} or the original
work of Zave~\cite{Zave:1976:harmonic}.

\begin{corollary}\label{cor:smallest-asy}
  The average number (expected value) of key comparisons in the comparison-optimal
  dual-pivot quickselect algorithm---it uses strategy ``Count''---when
  selecting the smallest or largest element is
  \begin{equation*}
    C^\ct_{n,1} = C^\ct_{n,n} = 
    \frac{9}{4} n + \frac{1}{12} (\log n)^2
    + \left(\frac{\gamma+\log 2}{6} + \frac{7}{3}\right) \log n
    + \Oh{1}
  \end{equation*}
  asymptotically as $n$ tends to infinity.
\end{corollary}

\begin{proof}[Proof of Proposition~\ref{pro:smallest-exact} and
  Corollary~\ref{cor:smallest-asy}]
  Again we use $\f{P}{z} = \f{P^\ct}{z}$ and write $\f{C^\ct_j}{z}$
  instead of $\f{C_j}{z}$. Solving the differential equation of
  Lemma~\ref{lem:ode-system} by
  Remark~\ref{rem:ode-qselect-system-solution} results in the
  generating function
  \begin{multline*}
    \f{C^\ct_1}{z} =
    \frac{9}{4} \frac{1}{(1-z)^{2}}
    + \frac{1}{12} \frac{\left(\log(1-z)\right)^{2}}{1-z}
    - \frac{1}{6} \frac{\f{L_2}{z}}{1-z}
    \\
    + \frac{7}{3} \frac{\log(1-z)}{1-z}
    - \frac{1}{18} \frac{1}{1-z} \log\biggl(\frac{1+z}{1-z}\biggr)
    - \frac{119}{54} \frac{1}{1-z}
    \\
    + \frac{1}{72}
    + \frac{1}{72} (1-z)
    + \frac{1}{144} (1-z)^2 \log\biggl(\frac{1+z}{1-z}\biggr)
    - \frac{2}{27} (1-z)^2.
  \end{multline*}
  To finish the proofs, we extract the coefficients, see also
  Appendix~\ref{sec:preparation}.
\end{proof}

The system of ordinary differential equations of
Lemma~\ref{lem:ode-system} can be solved iteratively. We calculate the
coefficients $C^\ct_{n,j}$ and $C^\ct_{n,n-j+1}$ with
$j\in\set{2,3,4}$ asymptotically in the following proposition. Exact
formul\ae{} and the proofs can be found in Appendix~\ref{sec:appendix-proofs}.

Note that it is possible to extend the result to $j=\Oh{1}$ by
collecting terms in each iteration; again a task for the full version
of this extended abstract.

\begin{proposition}\label{pro:j-4-smallest-asy}
  The average number (expected value) of key comparisons in the comparison-optimal
  dual-pivot quickselect algorithm---it uses strategy ``Count''---when
  selecting the first ($j=1$), 
  second ($j=2$), third ($j=3$) and fourth ($j=4$)
  smallest or largest element is
  \begin{equation*}
    C^\ct_{n,j} = C^\ct_{n,n-j+1} = 
    \frac{9}{4} n + \frac{1}{12} (\log n)^2
    +  \left(\frac{\gamma+\log 2}{6} + t_j\right) \log n
    + \Oh{1}
  \end{equation*}
  asymptotically as $n$ tends to infinity
  with
  \begin{align*}
    t_1 &= \tfrac{7}{3} = 2.333\ldots, &
    t_2 &= 1, \\
    t_3 &= - \tfrac{3}{10} = -0.3, &
    t_4 &= - \tfrac{29}{8} = -3.625.
  \end{align*}
\end{proposition}

Note that Proposition~\ref{pro:j-4-smallest-asy} superseds Corollary~\ref{cor:smallest-asy}. 
The proof of Proposition~\ref{pro:j-4-smallest-asy} can be found in
Appendix~\ref{sec:appendix-proofs}.


As in the section above, we state the corresponding formul\ae{} for the
``Clairvoyant'' partitioning strategy as well.

\begin{proposition}\label{pro:smallest-j-clairvoyant}
  The average number (expected value) of key comparisons in the
  dual-pivot quickselect algorithm with strategy ``Clairvoyant'' when
  selecting the smallest or largest element is
  \begin{multline*}
    C^\cv_{n,1} = C^\cv_{n,n} =
    \frac{9}{4} n
    -\frac{1}{12} \sum_{k=1}^{n-1} \frac{H_k}{n-k}
    + \frac{1}{6} \sum_{k=2}^n \frac{\Halt_{k-1}}{k}
    -\frac{41}{18} H_n
    -\frac{1}{18} \Halt_n
    + \frac{1}{108} \\
    -\frac{1}{72} \frac{\iverson*{$n$ odd}}{n-2}
    + \frac{1}{36} \frac{\iverson*{$n$ even}}{n-1}
    -\frac{1}{72} \frac{\iverson*{$n$ odd}}{n}  
  \end{multline*}
  This equals
  \begin{equation*}
    C^\cv_{n,1} = C^\cv_{n,n} =
    \frac{9}{4} n - \frac{1}{12} (\log n)^2
    + \left(-\frac{\gamma + \log 2}{6} + \frac{7}{3}\right) \log n
    + \Oh{1}
  \end{equation*}
  asymptotically as $n$ tends to infinity.
\end{proposition}

Again, the proof of Proposition~\ref{pro:smallest-j-clairvoyant} can be found in
Appendix~\ref{sec:appendix-proofs}.


And, again, as in the section above, we state the corresponding
formul\ae{} for the ``smaller pivot first'' partitioning strategy as
well; details of the proof can be found in Appendix~\ref{sec:p-first}.

\begin{proposition}\label{pro:smallest-spf}
  The average number (expected value) of key comparisons in the
  dual-pivot quickselect algorithm with strategy ``smaller pivot first'' when
  selecting the smallest or largest element is
  \begin{equation*}
    C^\spf_{n,1} = C^\spf_{n,n} =
    \frac{5}{2} n
    -\frac{8}{3} H_n
    + \frac{1}{18}.
  \end{equation*}
  This equals
  \begin{equation*}
    C^\spf_{n,1} = C^\spf_{n,n} =
    \frac{5}{2} n + \frac{8}{3} \log n + \frac{8}{3} \gamma
    - \frac{22}{9} - \frac{4}{3} n^{-1} + \Oh{n^{-2}}
  \end{equation*}
  asymptotically as $n$ tends to infinity.
\end{proposition}


\renewcommand{\MR}[1]{}
\bibliographystyle{amsplainurl}
\bibliography{bib/cheub}


\clearpage
\appendix


\section{Appendix to Section~\ref{sec:recurrence}}
\label{sec:appendix-recurrence}


Assuming $C_{n,j}=0$ if $n<0$ or $n<j$ or $j<1$ allows us to extend the
sums of Proposition~\ref{pro:recurrence} to
\begin{equation}\label{eq:recurrence:sums-extended}
  \begin{split}
  S_{n,j} &= \frac{1}{\binom{n}{2}} \sum_{s=0}^{n-1} (n-1-s) C_{s,j}, \\
  M_{n,j} &= \frac{1}{\binom{n}{2}} \sum_{m=0}^{n-2}
  \sum_{s=0}^{n-m-2} C_{m,j-s-1}, \\
  L_{n,j} &= \frac{1}{\binom{n}{2}} \sum_{\ell=0}^{n-1} (n-1-\ell) C_{\ell,n-j+1}.
\end{split}
\end{equation}


\begin{proof}[Proof of Proposition~\ref{pro:C-ode}]
  We use the recurrence of Proposition~\ref{pro:recurrence} with the
  extended sums~\eqref{eq:recurrence:sums-extended} to obtain
  \begin{multline*}
    n(n-1) C_{n,j} = n(n-1) P_n \iverson{1\leq j \leq n} \\
    + 2 \sum_{s=0}^{n-1} (n-1-s) C_{s,j}
    + 2 \sum_{m=0}^{n-2} \sum_{s=0}^{n-m-2} C_{m,j-s-1}
    + 2 \sum_{\ell=0}^{n-1} (n-1-\ell) C_{\ell,n-j+1}.   
  \end{multline*}
  Note that this recurrence is valid for $n=1$ as well (but only gives
  zero on both sides). We multiply by $z^{n-2} u^j$ and sum up over
  all $n\geq2$ and all $j$; we treat each summand separately, so we
  have an equation $\calC=\calP+\calS+\calM+\calL$.

  We obtain
  \begin{equation*}
    \calC
    = \sum_{j} \sum_{n\geq2} n(n-1) C_{n,j} z^{n-2} u^j
    = \frac{\partial^2}{\dd z^2}\f{C}{z,u}
  \end{equation*}
  and
  \begin{align*}
    \calP
    &= \sum_{n\geq2} n(n-1) P_n z^{n-2} \sum_{1\leq j\leq n} u^j
    = \sum_{n\geq2} n(n-1) P_n z^{n-2} u\frac{1-u^n}{1-u} \\
    &= \frac{u}{1-u} \left(\f{P''}{z} - u^2 \f{P''}{zu}\right).
  \end{align*}
  If $u=1$, then
  \begin{equation*}
    \calP
    = \sum_{n\geq2} n(n-1) P_n z^{n-2} \sum_{1\leq j\leq n} u^j
    = \sum_{n\geq2} n^2(n-1) P_n z^{n-2}
    = \frac{1}{z}  \bigl(z^2 \f{P''}{z}\bigr)'.
  \end{equation*}
  
  Next, we deal with $\calS$.
  We extend the sum by including $n=1$, then shift from $n-1$ to $n$, and get
  \begin{align*}
    \calS
    &= 2 \sum_{j} \sum_{n\geq2} \sum_{s=0}^{n-1} (n-1-s) C_{s,j} z^{n-2} u^j \\
    &= 2 \sum_{j} \sum_{n\geq1} \sum_{s=0}^{n-1} (n-1-s) C_{s,j} z^{n-2} u^j \\
    &= 2 \sum_{j} \sum_{n\geq0} \sum_{s=0}^n (n-s) z^{n-s-1} C_{s,j} z^s u^j.
  \end{align*}
  Rewriting the convolution to a product of generating functions yields
  \begin{equation*}
    \calS
    = 2 \Bigl( \sum_{n\geq0} n z^{n-1} \Bigr)
    \sum_{j} \sum_{n\geq0} C_{n,j} z^n u^j
    = 2 \Bigl(\frac{1}{1-z}\Bigr)' \f{C}{z,u}
    = \frac{2}{(1-z)^2} \f{C}{z,u}.
  \end{equation*}
  We proceed in a similar manner with $\calL$ and obtain
  \begin{align*}
    \calL
    &= 2 \sum_{j} \sum_{n\geq2} \sum_{\ell=0}^{n-1}
    (n-1-\ell) C_{\ell,n-j+1} z^{n-2} u^j \\
    &= 2 \sum_{j} \sum_{n\geq0} \sum_{\ell=0}^n
    (n-\ell) z^{n-\ell-1} C_{\ell,n-j+2} z^\ell u^j.
  \end{align*}
  We replace the sum over $j$ by the sum over $n+2-j$ and get
  \begin{align*}
    \calL
    &= 2 u^3 \sum_{j} u^{-j} \sum_{n\geq0} \sum_{\ell=0}^n
    (n-\ell) (zu)^{n-\ell-1} C_{\ell,j} (zu)^\ell \\
    &= 2 u^3 \Bigl( \sum_{n\geq0} n (zu)^{n-1} \Bigr)
    \sum_{j} \sum_{n\geq0} C_{n,j} (zu)^n u^{-j} \\
    &= 2 u^3 \left.\Bigl(\frac{1}{1-x}\Bigr)'\right\vert_{x=zu} \f{C}{zu,1/u}
    =  \frac{2 u^2}{(1-zu)^2} \f{C}{z,u}\!,
  \end{align*}
  where \eqref{eq:C-functional-eq} was used in the last step.

  To deal with the sum $\calM$, we proceed as follows. 
  Shifting the summation from $n-2$ to $n$ and substituting $t=j-s-1$ yields
  \begin{align*}
    \calM
    &= 2 \sum_{j} \sum_{n\geq0} \sum_{m=0}^n \sum_{s=0}^{n-m} C_{m,j-s-1} z^n u^j
    = 2 \sum_{t} \sum_{n\geq0} \sum_{m=0}^n \sum_{s=0}^{n-m} u^{s+1} C_{m,t} z^n u^t \\
    &= 2 \sum_{t} \sum_{n\geq0} \sum_{m=0}^n
    u \frac{1-u^{n-m+1}}{1-u} C_{m,t} z^n u^t.
  \end{align*}
  Some further rewriting gives
  \begin{align*}
    \calM
    &= \frac{2u}{1-u} \sum_{t} \sum_{n\geq0} \sum_{m=0}^n
    \bigl(z^{n-m} C_{m,t} z^m - u (zu)^{n-m} C_{m,t} z^m\bigr) u^t \\
    &= \frac{2u}{1-u} \sum_{t} \biggl(
    \biggl(\sum_{n\geq0} z^n\biggr)
    \biggl(\sum_{n\geq0} C_{n,t} z^n \biggr) - 
    u \biggl(\sum_{n\geq0} (zu)^n\biggr)
    \biggl( \sum_{n\geq0} C_{n,t} z^n \biggr) \biggr) u^t \\
    &= \frac{2u}{1-u}
    \biggl(\frac{1}{1-z} - \frac{u}{1-zu}\biggr) \f{C}{z,u}
    = \frac{2u}{(1-z)(1-zu)} \f{C}{z,u}.
  \end{align*}
  Note that $u=1$ results indeed in $\calM = 2\f{C}{z,1} / (1-z)^2$.

  As claimed, the overall result is $\calC=\calP+\calS+\calM+\calL$.
\end{proof}


\section{Notation and Preparation}
\label{sec:preparation}

The generating function of the harmonic numbers $H_m$ (Section~\ref{sec:notation})
is $-\log(1-z)/(1-z)$ and they satisfy the asymptotic
expansion
\begin{equation*}
  H_m = \log m + \gamma + \frac{1}{2m} - \frac{1}{12m^2} + \Oh{m^{-4}}
\end{equation*}
with the Euler--Mascheroni constant $\gamma=0.5772156649\ldots$\,.
Before we come to a variant of the harmonic numbers, we make a
short excursion to a generalization of the logarithm.
 
Let us denote the dilogarithm by $\Li{2}{x} = \sum_{m\geq1}
x^m/m^2$. It will be convenient to use a slightly modified function,
namely
\begin{equation*}
  \f{L_2}{z} = -\int_0^z \frac{\log (1+t)}{1-t} \dd t
  =
  - \Li[Big]{2}{\frac{1-z}{2}}
  + \log 2\, \log(1-z)
  + \frac{\pi^2}{12}
  - \frac{(\log 2)^2}{2}.
\end{equation*}
Note that using the functional equation
\begin{equation*}
  \Li{2}{x}+\Li{2}{1-x} = \frac{\pi^2}{6} - \log x \log(1-x)
\end{equation*}
(see, for example, Zagier~\cite{Zagier:2007:dilogarithm}) with
$x=(1+z)/2$ yields
\begin{align*}
  \f{L_2}{-z} &=
  \Li[Big]{2}{\frac{1-z}{2}}
  + \log \Bigl(\frac{1+z}{2}\Bigr) \log \Bigl(\frac{1-z}{2}\Bigr)
  + \log 2\, \log(1+z)
  - \frac{\pi^2}{12}
  - \frac{(\log 2)^2}{2} \\
  &= - \f{L_2}{z} + \log(1+z) \log(1-z).
\end{align*}

The alternating harmonic numbers $\Halt_m = \sum_{k=1}^m \frac{(-1)^k}{k}$ 
satisfy the generating function
\begin{equation*}
    \sum_{m\ge 1}\Halt_m z^m = -\frac{\log(1+z)}{(1-z)}.
\end{equation*}
Therefore $\Halt_{m-1}/m$ is the coefficient of $z^m$ in $\f{L_2}{z}$,
and, moreover, we obtain
\begin{equation*}
  \sum_{m\geq0} \sum_{k=2}^m \frac{\Halt_{k-1}}{k} z^m
  = \frac{\f{L_2}{z}}{1-z}.
\end{equation*}
As
\begin{equation*}
  \Halt_m = -\log 2 + \Oh{m^{-1}}
\end{equation*}
asymptotically as $m\to\infty$, we get
\begin{equation*}
  \sum_{k=2}^m \frac{\Halt_{k-1}}{k}
  = -H_m \log 2 + \Oh{1}
  = -\log 2 \log m + \Oh{1}.
\end{equation*}
Likewise the generating function $\f{L_2}{z}/(1-z)^2$ gives rise to the
coefficients
\begin{equation*}
  \sum_{k=2}^m (m-k+1) \frac{\Halt_{k-1}}{k}
  = (m+1) \sum_{k=2}^m \frac{\Halt_{k-1}}{k}
  - \sum_{k=2}^m \Halt_{k-1}
  = - m\log2 \log m + \Oh{m}.
\end{equation*}

During our calculations we need the generating functions
  \begin{align*}
    \sum_{m\geq0} \sum_{k=1}^{m-1} \frac{H_k}{m-k} z^m
    &= \frac{\log(1-z)^{2}}{1-z}
    \intertext{and}
    \sum_{m\geq0} \sum_{k=1}^{m-1} H_k H_{m-k} z^m
    &= \frac{\log(1-z)^{2}}{(1-z)^2}
  \end{align*}
as well.  


\section{More Proofs and Proof-Details}
\label{sec:appendix-proofs}


\begin{proof}[Proof of Theorem~\ref{thm:average-j-clairvoyant}]
  Solving the ordinary differential equation of
  Proposition~\ref{pro:C-ode} with $\f{P}{z} = \f{P^\cv}{z}$ yields
  the generating function
  \begin{multline*}
    \f{C^\cv}{z,1} =
    \frac{6}{(1-z)^3}
    - \frac{3 \left(\log(1-z)\right)^2}{20 (1-z)^2}
    + \frac{3 \f{L_2}{z}}{10 (1-z)^2}
    + \frac{41 \log(1-z)}{5 (1-z)^2}
    \\
    + \frac{9}{25 (1-z)^2} \log\biggl(\frac{1+z}{1-z}\biggr)
    - \frac{529}{125 (1-z)^2}
    \\
    - \frac{1}{8 (1-z)} \log\biggl(\frac{1+z}{1-z}\biggr)
    - \frac{1411}{800 (1-z)}
    \\
    - \frac{11}{1200}
    - \frac{11}{1600} (1-z)
    - \frac{11}{1600} (1-z)^2
    \\
    - \frac{11}{3200} (1-z)^3 \log\biggl(\frac{1+z}{1-z}\biggr)
    + \frac{7}{375} (1-z)^3
  \end{multline*}
  from which the coefficients can be extracted.

  Solving an ordinary differential equation obtained from
  Lemma~\ref{lem:ode-system} with $\f{P}{z} = \f{P^\cv}{z}$ yields the
  generating function
  \begin{multline*}
    \f{C^\cv_1}{z} =
    \frac{9}{4 (1-z)^2}
    - \frac{\left(\log(1-z)\right)^2}{12 (1-z)}
    + \frac{\f{L_2}{z}}{6 (1-z)}
    - \frac{121}{54 (1-z)}
    \\
    + \frac{7 \log(1-z)}{3 (1-z)}
    + \frac{1}{18 (1-z)} \log\biggl(\frac{1+z}{1-z}\biggr)
    \\
    - \frac{1}{72}
    - \frac{1}{72} (1-z)
    - \frac{1}{144} (1-z)^2 \log\biggl(\frac{1+z}{1-z}\biggr)
    + \frac{1}{54} (1-z)^2
  \end{multline*}
  from which again the coefficients can be extracted.
\end{proof}


\begin{proof}[Proof of Lemma~\ref{lem:ode-system}]
  We use the notation $\calC=\calP+\calS+\calM+\calL$ 
  of the proof of Proposition~\ref{pro:C-ode}.
  It is easy to see that $[u^j] \calC = \f{C_j''}{z}$. We have 
  \begin{equation*}
    [u^j] \calP = \f{P''}{z} - \sum_{n<j} n(n-1) P_n z^{n-2}
  \end{equation*}
  and $[u^j] \calS = 2/(1-z)^2 \f{C_j}{z}$. The remaining two quantities are
  \begin{equation*}
    [u^j] \calM = \frac{2}{1-z} [u^{j-1}] \frac{1}{1-zu} \f{C}{z,u}
    = \frac{2}{1-z} \sum_{k=0}^{j-1} \f{C_k}{z} z^{j-k-1}
  \end{equation*}
  and
  \begin{equation*}
    [u^j] \calL = 2 [u^{j-2}] \frac{1}{(1-zu)^2} \f{C}{z,u}
    = 2 \sum_{k=0}^{j-2} \f{C_k}{z} (j-k-1) z^{j-k-2}.
  \end{equation*}
  Rewriting gives the result that we wanted to show.
\end{proof}


\begin{proof}[Proof of Remark~\ref{rem:ode-qselect-system-solution}]
  This proof is based on Hennequin~\cite{Hennequin:1991:analy} and Wild~\cite{Wild2013}.
  (See also \cite{Aumueller-Diezfelbinger-Heuberger-Krenn-Prodinger:2016:quicksort-paths-arxiv}.)

  By setting $\f{(\theta f)}{z} = (1-z) \f{f'}{z}$ we have
  \begin{equation*}
    \f{((\theta^2+\theta-2)C)}{z}
    = (1-z)^2 \f{C''}{z} - 2\f{C}{z}
    = (1-z)^2 \f{Q}{z}.
  \end{equation*}
  As $\theta^2+\theta-2=(\theta-1)(\theta+2)$, we first solve for
  $D=(\theta+2)C$ in
  \begin{equation*}
     \f{((\theta-1)D)}{z} = (1-z^2) Q.
  \end{equation*}
  The left hand side equals
  \begin{equation*}
    \f{((\theta-1)D)}{z}
    = (1-z)\f{D'}{z} - \f{D}{z}
    = \left((1-z)\f{D}{z}\right)'\!,
  \end{equation*}
  and we have $\f{D}{0} = \f{C'}{0} + 2\f{C}{0} = 0$, so
  \begin{equation*}
    \f{D}{z} = (1-z)^{-1} \int_0^z (1-s)^2 \f{Q}{z} \dd s.
  \end{equation*}
  As a second step, we solve
  \begin{equation*}
    (1-z) \f{C'}{z} + 2\f{C}{z} = \f{((\theta+2)C)}{z} = \f{D}{z}.
  \end{equation*}
  Multiplying by $(1-z)^{-3}$ yields
  \begin{equation*}
    \left( (1-z)^{-2} \f{C}{z} \right)'
    = (1-z)^{-2} \f{C'}{z} + 2(1-z)^{-3} \f{C}{z} = (1-z)^{-3} \f{D}{z}
  \end{equation*}
  which, together with $\f{C}{0} = 0$ results in
  \eqref{eq:ode-qselect-solution}.
\end{proof}


\begin{proposition}\label{pro:j-4-smallest-exact}
  The average number (expected value) of key comparisons in the comparison-optimal
  dual-pivot quickselect algorithm when selecting the second, third and fourth
  smallest or largest element is
  \begin{multline*}
    C^\ct_{n,2} = C^\ct_{n,n-1} =
\frac{9}{4} n
+ \frac{1}{12} \sum_{k=1}^{n-1} \frac{H_k}{n-k}
-\frac{1}{6} \sum_{k=2}^n \frac{\Halt_{k-1}}{k}
-\frac{8}{9} H_n
-\frac{1}{9} \Halt_n \\
-\frac{755}{216} 
-\frac{1}{12} \sum_{k=1}^{n-1} \frac{1}{k(n-k)}
+ \frac{1}{6} \frac{\Halt_{n-1}}{n} \\
-\frac{1}{144} \frac{\iverson*{$n$ even}}{n-3}
-\frac{1}{144} \frac{\iverson*{$n$ odd}}{n-2}
+ \frac{5}{144} \frac{\iverson*{$n$ even}}{n-1}
+ \frac{7}{3} \frac{\iverson*{$n$ even}}{n}
+ \frac{325}{144} \frac{\iverson*{$n$ odd}}{n}
  \end{multline*}
  and
  \begin{multline*}
    C^\ct_{n,3} = C^\ct_{n,n-2} =
\frac{9}{4} n
+ \frac{1}{12} \sum_{k=1}^{n-1} \frac{H_k}{n-k}
-\frac{1}{6} \sum_{k=2}^n \frac{\Halt_{k-1}}{k}
+ \frac{11}{18} H_n
-\frac{14}{45} \Halt_n \\
-\frac{383}{54} 
-\frac{1}{12} \sum_{k=1}^{n-1} \frac{1}{k(n-k)}
-\frac{1}{12} \sum_{k=1}^{n-2} \frac{1}{k(n-k-1)}
+ \frac{1}{6} \frac{\Halt_{n-1}}{n}
+ \frac{1}{6} \frac{\Halt_{n-2}}{n-1} \\
+ \frac{1}{720} \frac{\iverson*{$n$ odd}}{n-4}
+ \frac{1}{720} \frac{\iverson*{$n$ even}}{n-3}
+ \frac{2}{3} \frac{\iverson*{$n$ even}}{n-2}
+ \frac{541}{720} \frac{\iverson*{$n$ odd}}{n-2} \\
+ \frac{671}{720} \frac{\iverson*{$n$ even}}{n-1}
+ 1 \frac{\iverson*{$n$ odd}}{n-1}
+ \frac{5}{3} \frac{\iverson*{$n$ even}}{n}
+ \frac{433}{360} \frac{\iverson*{$n$ odd}}{n}
  \end{multline*}
  and
  \begin{multline*}
    C^\ct_{n,4} = C^\ct_{n,n-3} =
\frac{9}{4} n
+ \frac{1}{12} \sum_{k=1}^{n-1} \frac{H_k}{n-k}
-\frac{1}{6} \sum_{k=2}^n \frac{\Halt_{k-1}}{k}
+ \frac{19}{9} H_n
-\frac{1}{2} \Halt_n \\
-\frac{11743}{1080} 
-\frac{1}{4} \sum_{k=1}^{n-2} \frac{1}{k(n-k-1)}
+ \frac{1}{2} \frac{\Halt_{n-2}}{n-1}
+ \frac{1}{720} \frac{\iverson*{$n$ even}}{n-5}
-\frac{1}{144} \frac{\iverson*{$n$ odd}}{n-4} \\
-\frac{13}{36} \frac{\iverson*{$n$ even}}{n-3}
-\frac{1}{3} \frac{\iverson*{$n$ odd}}{n-3}
+ 7 \frac{\iverson*{$n$ even}}{n-2}
+ \frac{65}{9} \frac{\iverson*{$n$ odd}}{n-2} \\
-\frac{1105}{144} \frac{\iverson*{$n$ even}}{n-1}
-\frac{22}{3} \frac{\iverson*{$n$ odd}}{n-1}
+ \frac{37}{10} \frac{\iverson*{$n$ even}}{n}
+ \frac{377}{144} \frac{\iverson*{$n$ odd}}{n}.    
  \end{multline*}
\end{proposition}


\begin{proof}[Proof of Propositions~\ref{pro:j-4-smallest-asy}
  and~\ref{pro:j-4-smallest-exact}]

  Solving iteratively the first four ordinary differential equation
  obtained from Lemma~\ref{lem:ode-system} with $\f{P}{z} =
  \f{P^\ct}{z}$ yields the following generating functions.
  For $j=2$ we obtain
  \begin{multline*}
    \f{C^\ct_2}{z} =
    \frac{9}{4 (1-z)^2}
    + \frac{\left(\log(1-z)\right)^2}{12 (1-z)}
    - \frac{\f{L_2}{z}}{6 (1-z)}
    \\
    + \frac{\log(1-z)}{(1-z)}
    + \frac{1}{9 (1-z)} \log\biggl(\frac{1+z}{1-z}\biggr)
    - \frac{1241}{216 (1-z)}
    \\
    - \frac{1}{12} \left(\log(1-z)\right)^2
    + \frac{1}{6} \f{L_2}{z}
    - \frac{7}{3} \log(1-z)
    - \frac{1}{36} \log\biggl(\frac{1+z}{1-z}\biggr)
    + \frac{91}{27}    
    \\
    - \frac{1}{48} (1-z)
    - \frac{1}{72} (1-z)^2 \log\biggl(\frac{1+z}{1-z}\biggr)
    + \frac{79}{432} (1-z)^2
    \\
    + \frac{1}{288} (1-z)^3 \log\biggl(\frac{1+z}{1-z}\biggr)
    - \frac{1}{27} (1-z)^3.
  \end{multline*}
  The generating functions for $j=3$ is
  \begin{multline*}
    \f{C^\ct_3}{z} =
    \frac{9}{4 (1-z)^2}
    + \frac{\left(\log(1-z)\right)^2}{12 (1-z)}
    - \frac{\f{L_2}{z}}{6 (1-z)}
    \\
    - \frac{3 \log(1-z)}{10 (1-z)}
    + \frac{14}{45 (1-z)} \log\biggl(\frac{1+z}{1-z}\biggr)
    - \frac{1009}{108 (1-z)}
    \\
    - \frac{1}{6} \left(\log(1-z)\right)^2
    + \frac{1}{3} \f{L_2}{z}
    - \frac{10}{3} \log(1-z)
    - \frac{2}{9} \log\biggl(\frac{1+z}{1-z}\biggr)
    + \frac{5149}{540}
    \\
    + \frac{1}{12} (1-z) \left(\log(1-z)\right)^2
    - \frac{1}{6} (1-z) \f{L_2}{z}
    + \frac{7}{3} (1-z) \log(1-z)
    \\
    - \frac{1}{18} (1-z) \log\biggl(\frac{1+z}{1-z}\biggr)
    - \frac{4601}{2160} (1-z)
    \\
    - \frac{55}{72} (1-z)^2 \log(1-z)
    + \frac{7}{144} (1-z)^2 \log\biggl(\frac{1+z}{1-z}\biggr)
    - \frac{193}{540} (1-z)^2
    \\
    - \frac{1}{288} (1-z)^3 \log\biggl(\frac{1+z}{1-z}\biggr)
    + \frac{113}{2160} (1-z)^3
    \\
    - \frac{1}{1440} (1-z)^4 \log\biggl(\frac{1+z}{1-z}\biggr)
    + \frac{1}{135} (1-z)^4,
  \end{multline*}
  and the generating function for $j=4$ is
  \begin{multline*}
    \f{C^\ct_4}{z} =
    \frac{9}{4 (1-z)^2}
    + \frac{\left(\log(1-z)\right)^2}{12 (1-z)}
    - \frac{\f{L_2}{z}}{6 (1-z)}
    \\
    - \frac{29 \log(1-z)}{18 (1-z)}
    + \frac{1}{2 (1-z)} \log\biggl(\frac{1+z}{1-z}\biggr)
    - \frac{14173}{1080 (1-z)}
    \\
    - \frac{1}{4} \left(\log(1-z)\right)^2
    + \frac{1}{2} \f{L_2}{z}
    - \frac{91}{30} \log(1-z)
    - \frac{37}{60} \log\biggl(\frac{1+z}{1-z}\biggr)
    + \frac{445}{24}
    \\
    + \frac{1}{4} (1-z) \left(\log(1-z)\right)^2
    - \frac{1}{2} (1-z) \f{L_2}{z}
    \\
    + \frac{17}{3} (1-z) \log(1-z)
    - \frac{1373}{180} (1-z)
    \\
    - 6 (1-z)^2 \log(1-z)
    + \frac{1}{18} (1-z)^2 \log\biggl(\frac{1+z}{1-z}\biggr)
    - \frac{4687}{1080} (1-z)^2
    \\
    - \frac{1}{3} (1-z)^3 \log(1-z)
    + \frac{1}{48} (1-z)^3 \log\biggl(\frac{1+z}{1-z}\biggr)
    + \frac{3089}{720} (1-z)^3
    \\
    - \frac{1}{720} (1-z)^4
    - \frac{1}{1440} (1-z)^5 \log\biggl(\frac{1+z}{1-z}\biggr)
    + \frac{1}{135} (1-z)^5.
  \end{multline*}
  Extracting the coefficients yields the desired results.
\end{proof}


\section{Partitioning Strategy: Smaller Pivot First}
\label{sec:p-first}


As mentioned at the end of Section~\ref{sec:average-j},
we include the expected value/average number of key comparisons for
dual-pivot quickselect with the partitioning strategy ``smaller pivot
first'' for completeness.

\begin{proposition}
  Classifying the elements of a list of $n$ elements with the
  dual-pivot partitioning strategy where the first comparison of each
  element is always with the smaller pivot (``smaller pivot first'')
  needs on average
  \begin{equation*}
    P^\spf_n = \frac{5}{3}n-\frac{7}{3},
  \end{equation*}
  $n\geq2$,
  key comparisons. The corresponding generating function is
  \begin{equation*}
    \f{P^\spf}{z} = \sum_{n\geq0} P^\spf_n z^n
    = \frac{5}{3(1-z)^2} - \frac{4}{1-z} - \frac{2}{3}(1-z) + 3.
  \end{equation*}
\end{proposition}

\begin{proof}
  We fix the two pivot elements $p$ and $q$; one comparison is needed
  to ensure $p<q$. To classify a small element, we need one
  comparison, and to classify a medium or large element, we need
  two. Summing up yields
  \begin{equation*}
    \sum_{s+m+\ell=n-2} (s+2m+2\ell) = \binom{n}{2} \frac{5}{3}(n-2).
  \end{equation*}
  The result follows by normalizing by $\binom{n}{2}$ of all
  possibilities $s+m+\ell=n-2$ and adding $1$.
\end{proof}


\begin{proof}[Proof of Proposition~\ref{pro:average-j-spf}]
  Solving the ordinary differential equation of
  Proposition~\ref{pro:C-ode} with $\f{P}{z} = \f{P^\spf}{z}$ yields
  the generating function
  \begin{equation*}
    \f{C^\spf}{z,1} = 
    \frac{20}{3(1-z)^3}
    + \frac{44 \log(1-z)}{5(1-z)^2} 
    - \frac{116}{25(1-z)^2}
    - \frac{2}{1-z} 
    - \frac{2}{75} (1-z)^3
  \end{equation*}
  from which the coefficients can be extracted.
\end{proof}


\begin{proof}[Proof of Proposition~\ref{pro:smallest-spf}]
  Solving the corresponding ordinary differential equation obtained from
  Lemma~\ref{lem:ode-system} with $\f{P}{z} = \f{P^\spf}{z}$ yields the
  generating function
  \begin{equation*}
    \f{C^\spf_1}{z} =
    \frac{5}{2(1-z)^2}
    + \frac{8 \log(1-z)}{3 (1-z)}
    - \frac{22}{9(1-z)}
    -\frac{1}{18} (1-z)^2
  \end{equation*}
  from which the coefficients can be extracted.  
\end{proof}


\end{document}